\documentclass[11pt,a4paper]{article}
\usepackage[T1]{fontenc}
\usepackage[utf8]{inputenc}
\usepackage{lmodern}
\usepackage[margin=28mm]{geometry}
\usepackage{amsmath,amssymb,amsthm,mathtools}
\usepackage{microtype}
\usepackage{enumitem}
\usepackage[hidelinks]{hyperref}

\hypersetup{
  pdftitle={Sharp lower bounds for bisections of graphs with bounded maximum degree},
  pdfauthor={Qiwen Guo and Gregory Gutin},
  pdfsubject={A Gallai-Edmonds and Hall-deficiency proof of sharp unweighted bisection bounds},
  pdfkeywords={maximum bisection, maximum degree, Gallai-Edmonds decomposition, Hall's theorem, degree deficiency}
}

\newtheorem{theorem}{Theorem}[section]
\newtheorem{lemma}[theorem]{Lemma}
\newtheorem{proposition}[theorem]{Proposition}
\theoremstyle{definition}

\theoremstyle{remark}

\numberwithin{equation}{section}
\newcommand{\bis}{\operatorname{bis}}
\newcommand{\ceil}[1]{\left\lceil #1\right\rceil}
\newcommand{\floor}[1]{\left\lfloor #1\right\rfloor}
\newcommand{\cD}{\mathcal D}
\newcommand{\cE}{\mathcal E}
\newcommand{\cO}{\mathcal O}
\newcommand{\cX}{\mathcal X}
\setlist[enumerate]{label=(\roman*),leftmargin=2em,itemsep=2pt,topsep=5pt}
\title{Sharp lower bounds for bisections of graphs\\[2pt]with bounded maximum degree}
\author{Qiwen Guo and Gregory Gutin\\ Royal Holloway, University of London, UK}
\date{}

\begin{document}
\maketitle

\begin{abstract}
A bisection of a graph is a partition of its vertex set into two parts
whose sizes differ by at most one. We prove that every graph with $m$
edges and maximum degree at most $k$ has a bisection of size at least
$\frac{k+1}{2k}m$ when $k$ is odd, and at least
$\frac{k+2}{2(k+1)}m$ when $k$ is even. Both bounds are sharp, with
equality for complete graphs. This establishes the unweighted case
of the bisection conjecture of Gerke, Gutin, Yeo and Zhou (arXiv 2024), removes the
additive error from the corresponding bounds of Lee, Loh and Sudakov (JCTB 2013),
and extends the odd-degree regular case proved by Bollob\'as and Scott (JGT 2004)
to arbitrary graphs of bounded maximum degree. The proof combines
the matching-and-pairing method with the Gallai--Edmonds decomposition.
Hall's theorem controls the number of uncovered factor-critical
components whose internal degrees are all even. A boundary-edge count
bounds this obstruction by the total degree deficiency, while the
remaining uncovered components supply the required pairing gain.
\end{abstract}


\section{Introduction}\label{sec:introduction}

All graphs considered in this paper are finite, simple and unweighted.
For a graph $G$, write $n=|V(G)|$, $m=|E(G)|$, and let $\Delta(G)$
denote its maximum degree. A \emph{cut} is the set of edges between
the two parts of a partition $(X,Y)$ of $V(G)$. The cut is a
\emph{bisection} if $\bigl||X|-|Y|\bigr|\le1$. Let $\bis(G)$ be the maximum number of edges in a bisection of $G$.

The maximum-cut and maximum-bisection problems ask for partitions
with many crossing edges, but the balance requirement makes the
second problem more restrictive. Pairing vertices, leaving at most
one unpaired vertex, and independently placing the vertices of each
pair on opposite sides gives the basic bound $\bis(G)\ge m/2$.
A natural extremal question is whether the sharp degree-dependent
coefficients for unrestricted cuts also hold for bisections; see
\cite{GGYZ} and the references therein.

Lee, Loh and Sudakov~\cite{LLS} proved that every graph with maximum
degree at most $k$ has a bisection satisfying
\begin{equation}\label{eq:LLS}
\bis(G)\ge
\begin{cases}
\displaystyle\frac{k+1}{2k}\,m-\frac{k(k+1)}4,
   &\text{if $k$ is odd},\\[6pt]
\displaystyle\frac{k+2}{2(k+1)}\,m-\frac{k(k+2)}4,
   &\text{if $k$ is even}.
\end{cases}
\end{equation}
Their proof uses an almost-balanced partition obtained from an
equitable coloring and then restores balance. Bollob\'as and
Scott~\cite{BS} had proved the exact bound
$\bis(G)\ge\frac{k+1}{2k}m$ for $k$-regular graphs when $k$ is odd;
Lee, Loh and Sudakov also give a proof of this regular-graph result.

Gerke, Gutin, Yeo and Zhou~\cite[Conjecture~1.1]{GGYZ} conjectured
that the two coefficients in~\eqref{eq:LLS} hold without the additive
error, even for graphs with arbitrary nonnegative real edge weights.
Here the size of a cut is replaced by its total edge weight.
They proved the weighted assertion when $k$ is even, for graphs
of chromatic-index Class~1, and for subcubic graphs. A graph is of
\emph{Class~1} when its chromatic index equals its maximum degree;
it is \emph{subcubic} when its maximum degree is at most three.
The unweighted case not covered by these results is the exact
odd-$k$ bound without a regularity or Class~1 assumption or $k=3.$

Our main result settles the unweighted case of this conjecture.
\begin{theorem}\label{thm:main}
Let $k$ be a positive integer, let $G$ be a graph with
$\Delta(G)\le k$, and let $m=|E(G)|$. Then
\begin{equation}\label{eq:main}
\bis(G)\ge
\begin{cases}
\displaystyle\frac{k+1}{2k}\,m,&\text{if $k$ is odd},\\[6pt]
\displaystyle\frac{k+2}{2(k+1)}\,m,&\text{if $k$ is even}.
\end{cases}
\end{equation}
For each $k$, the corresponding coefficient is best possible.
\end{theorem}

Equality is attained by $K_{k+1}$. The even case follows from the
odd case applied with $k+1$, so the new assertion is the remaining
odd-$k$ case. Isolated vertices and disconnected graphs are allowed.

\subsection*{Proof strategy}

We use the matching-and-pairing method of Lee, Loh and
Sudakov~\cite[Section~3.2]{LLS}, in a form stated and proved in
Section~\ref{sec:pairing}. Separating the end-vertices of each matching
edge gives a gain of $1/2$ over $m/2$. Certain
pairs of exposed vertices give the same gain, using the parity of
their adjacencies to matching edges.

The Gallai--Edmonds decomposition allows us to choose the exposed
vertices component by component. In any uncovered component having
a vertex of odd internal degree, we expose such a vertex; it can be
paired advantageously with any other exposed vertex. We give priority
to covering the components whose internal degrees are all even.
Hall's theorem shows that the minimum number $t$ of these components
left uncovered is a Hall deficiency. A direct count of their internal
and boundary edges gives
\[
kt\le\Phi_k(G):=kn-2m
\qquad\text{when $k$ is odd and $\Delta(G)\le k$}.
\]
For even $n$, the pairing argument gives
$\bis(G)\ge m/2+(n-t)/4$, which implies the required bound.

Section~\ref{sec:preliminaries} gives the matching definitions and
standard theorems used below. Section~\ref{sec:pairing} proves the
pairing inequality, and Section~\ref{sec:deficiency} establishes the
Hall-deficiency estimate. Section~\ref{sec:structural} gives the
structural bisection bound. Section~\ref{sec:mainproof} proves
Theorem~\ref{thm:main}, and Section~\ref{sec:conclusion} discusses the
scope of the argument.

\section{Matching preliminaries}\label{sec:preliminaries}

For $v\in V(G)$, let $d_G(v)$ and $N_G(v)$ denote its degree and
neighborhood. For disjoint vertex sets $U,Z$, let $E_G(U,Z)$ be the set
of edges with one end-vertex in each set, and put
$e_G(U,Z)=|E_G(U,Z)|$. For a subgraph $T$, put $e(T)=|E(T)|$ and let
$d_T(v)$ denote degree of $v$ in $T$.

A \emph{matching} is a set of pairwise vertex-disjoint edges. It
\emph{covers} or \emph{saturates} a vertex set if every vertex in the
set is incident with a matching edge. A \emph{maximum matching} has
largest cardinality, denoted by $\nu(G)$, and a \emph{perfect matching}
covers $V(G)$. A vertex not covered by a specified matching is
\emph{exposed}, or \emph{uncovered}. A \emph{near-perfect matching}
leaves exactly one vertex exposed. The empty matching is perfect in the graph with no vertices. 
A nonempty graph $T$ is
\emph{factor-critical} if $T-v$ has a perfect matching for each
$v\in V(T)$. Thus a factor-critical graph has odd order.

When $\Delta(G)\le k$, define the \emph{total degree deficiency} by
\begin{equation}\label{eq:deficiency}
\Phi_k(G):=\sum_{v\in V(G)}\bigl(k-d_G(v)\bigr)=kn-2m.
\end{equation}
Every summand is nonnegative.

\subsection{Hall's theorem and prescribed coverage}

For a bipartite graph $B$ with parts $L,Q$ and a set $S$ contained
in either part, let $N_B(S)$ denote its neighborhood in the other
part. We use Hall's theorem in its usual matching form.

\begin{theorem}[Hall~\cite{Hall}]\label{thm:Hall}
A bipartite graph $B$ with parts $L,Q$ has a matching saturating $Q$
if and only if $|N_B(S)|\ge |S|$ for every $S\subseteq Q$.
\end{theorem}

The following consequence includes the coverage requirement needed
in the proof.
\begin{lemma}[Hall deficiency with prescribed coverage]\label{lem:preferred}
Let $B$ be a bipartite graph with parts $L,Q$, and suppose that $B$
has a matching saturating $L$. For $Q_0\subseteq Q$, put
\begin{equation}\label{eq:preferreddef}
\delta:=\max_{S\subseteq Q_0}\bigl(|S|-|N_B(S)|\bigr).
\end{equation}
Then $B$ has a matching saturating $L$ and leaving exactly $\delta$
vertices of $Q_0$ uncovered. Every matching of $B$ leaves at least
$\delta$ vertices of $Q_0$ uncovered.
\end{lemma}

\begin{proof}
Any matching covers at most $|N_B(S)|$ vertices of $S$, so it leaves
at least $|S|-|N_B(S)|$ vertices of $Q_0$ uncovered. This proves the
lower bound. Notice that $\delta\ge0$, since $S=\varnothing$ is
permitted in~\eqref{eq:preferreddef}.

To attain the bound without yet requiring saturation of $L$, take
$B[L\cup Q_0]$ and add $\delta$ new vertices to the part $L$, each
adjacent to every vertex of $Q_0$. For every nonempty
$S\subseteq Q_0$, its new neighborhood has size
$|N_B(S)|+\delta\ge |S|$. Hall's theorem gives a matching saturating
$Q_0$. Deleting the edges incident with the new vertices gives a
matching $J$ in $B[L\cup Q_0]$ leaving at most $\delta$ vertices
of $Q_0$ uncovered, and hence exactly $\delta$ by the lower bound.

Regard $J$ as a matching of $B$. If it does not saturate $L$, compare
it with a matching $K$ that does. The symmetric difference
$J\triangle K=(J\setminus K)\cup(K\setminus J)$ consists of paths
and even cycles whose edges alternate between the two matchings.
Since $|K|>|J|$, some path has one more edge of $K$ than of $J$;
its end-vertices are exposed by $J$. Exchanging the edges along this
path increases $|J|$ by one and preserves every vertex already
covered by $J$. Repeat until $L$ is saturated. The final matching
still covers all the previously covered vertices of $Q_0$, so it leaves at most $\delta$ vertices of $Q_0$ uncovered. By the lower bound, it leaves at least $\delta$ vertices of $Q_0$ uncovered. It therefore has the required properties.
\end{proof}

\subsection{The Gallai--Edmonds decomposition}

Define
\begin{equation}\label{eq:GEsets}
\begin{aligned}
D&:=\{v\in V(G):v\text{ is exposed by some maximum matching of }G\},\\
A&:=\left(\bigcup_{v\in D}N_G(v)\right)\setminus D,\\
R&:=V(G)\setminus(D\cup A).
\end{aligned}
\end{equation}
We call $(D,A,R)$ the \emph{Gallai--Edmonds decomposition} of $G$.
Let $\cD$ be the family of connected components of $G[D]$, viewed
as induced subgraphs. By definition, there are no edges between
$D$ and $R$. We use the following standard form of the structure
theorem; see, for example, West~\cite{West}.

\begin{theorem}[Gallai--Edmonds]\label{thm:GE}
For the sets in~\eqref{eq:GEsets}, the following hold.
\begin{enumerate}
\item Every $T\in\cD$ is factor-critical.
\item The graph $G[R]$ has a perfect matching.
\item Every maximum matching consists of a perfect matching of
$G[R]$, a near-perfect matching within each $T\in\cD$, and edges
joining all vertices of $A$ to the internally exposed vertices of
$|A|$ distinct components in $\cD$.
\end{enumerate}
In particular,
\begin{equation}\label{eq:GEcardinality}
\nu(G)=\frac{n-|\cD|+|A|}{2}.
\end{equation}
\end{theorem}

Define the \emph{component-incidence graph} $B$ to be the bipartite
graph with parts $A$ and $\cD$, where
\begin{equation}\label{eq:incidence}
aT\in E(B)\quad\Longleftrightarrow\quad
E_G(\{a\},V(T))\ne\varnothing,\quad \text{for}\quad a\in A \quad \text{and} \quad T\in \cD.
\end{equation}
Thus a vertex on the $\cD$ side represents an entire component of
$G[D]$. Theorem~\ref{thm:GE}(iii) implies that $B$ has a matching
saturating $A$.

Conversely, any matching $J$ of $B$ saturating $A$ can be lifted to
a maximum matching of $G$. For each $aT\in J$, choose an actual edge
$av$ with $v\in V(T)$ and a perfect matching of $T-v$. In each
component not covered by $J$, expose any chosen vertex $w$ and take
a perfect matching of $T-w$. Add a perfect matching of $G[R]$.
These choices are possible by Theorem~\ref{thm:GE}(i)--(ii).
The resulting matching has
\[
\frac{|R|}{2}+\sum_{T\in\cD}\frac{|V(T)|-1}{2}+|A|
=\frac{n-|\cD|+|A|}{2}
\]
edges, so it is maximum by~\eqref{eq:GEcardinality}. In particular,
the exposed vertex in each component not covered by $J$ may be
chosen independently.

\section{A pairing inequality}\label{sec:pairing}

We use the pairing argument from~\cite[Section~3.2]{LLS}, expressed
directly through the parity of adjacency counts. 

\begin{lemma}[Pairing inequality]\label{lem:pairing}
Suppose that $G$ has even order, and let $M$ be a maximum matching
of size $s$, with exposed vertex set $W$. Let $P_1,\ldots,P_q$
be pairwise disjoint two-element subsets of $W$. Suppose that for
every $i$ there is an edge $xy\in M$ such that
\begin{equation}\label{eq:oddadjacency}
e_G(P_i,\{x,y\})\quad\text{is odd}.
\end{equation}
Then
\begin{equation}\label{eq:pairbound}
\bis(G)\ge\frac m2+\frac{s+q}{2}.
\end{equation}
\end{lemma}

\begin{proof}
The set $W$ is independent, since an edge joining two exposed
vertices could be added to $M$. Order the matching edges as
$e_1,\ldots,e_s$ and regard each of them as a pair of vertices.
For each specified pair $P_i$, let $j(i)$ be the first index such
that $e_G(P_i,V(e_{j(i)}))$ is odd. Insert $P_i$ immediately after
$e_{j(i)}$, in any order if several pairs have the same index.
All earlier matching pairs have an even number of edges to $P_i$,
and the matching pair at index $j(i)$ has an odd number. Since
there are no edges within $W$, the total number of edges from $P_i$
to the preceding vertices is odd. Pair the remaining vertices of
$W$ arbitrarily and append these pairs; their number is even because
$n-2s-2q$ is even.

Process the resulting ordered partition of $V(G)$ into pairs.
At each step place one vertex of the current pair in each part,
choosing the placement with more edges crossing to earlier vertices.
If $b$ edges join the current pair to earlier vertices, the two
placements have crossing counts summing to $b$, so the better
placement contributes at least $\ceil{b/2}$ of these edges.
Any edge within the pair crosses in either placement.

For a matching pair, the newly counted edges number $b+1$, and
at least $b/2+1=(b+1)/2+1/2$ of them cross. For a specified pair
$P_i$, there is no internal edge and $b$ is odd, so at least
$(b+1)/2$ of its $b$ newly counted edges cross. Every other pair
contributes at least half its newly counted edges. Each edge of
$G$ is counted exactly once. Summing the gains of $1/2$ from the
$s$ matching pairs and $q$ specified pairs proves~\eqref{eq:pairbound}.
Both parts contain one vertex from every pair, so they have equal size.
\end{proof}

Lemma \ref{lem:pairing} can also be proved using the generic low bound of $\bis(G)$ in Lemma 2.2 in \cite{GGYZ}.

\section{Hall deficiency and degree deficiency}\label{sec:deficiency}

Use the Gallai--Edmonds decomposition and component-incidence graph
$B$ from Section~\ref{sec:preliminaries}. Partition $\cD$ into
\begin{equation}\label{eq:componentclasses}
\cE:=\{T\in\cD:d_T(v)\text{ is even for every }v\in V(T)\},
\qquad
\cO:=\cD\setminus\cE.
\end{equation}
The distinction concerns \emph{internal degrees}, not component
orders: every component in $\cD$ has odd order. Single-vertex
components belong to $\cE$; each component in $\cO$ has a vertex
of odd internal degree.

Define
\begin{equation}\label{eq:tdef}
t=t(G):=\max_{\cX\subseteq\cE}
\bigl(|\cX|-|N_B(\cX)|\bigr).
\end{equation}
We call $t$ the \emph{Hall deficiency of the even-degree components}.
It is a nonnegative integer, with $t=0$ if $\cE=\varnothing$.
Lemma~\ref{lem:preferred}, applied with $L=A$, $Q=\cD$ and
$Q_0=\cE$, shows that $B$ has a matching saturating $A$ and leaving
exactly $t$ components of $\cE$ uncovered.

\begin{lemma}[Degree-deficiency bound]\label{lem:deficiency}
Let $k$ be an odd positive integer and let $\Delta(G)\le k$.
Then, for $t$ defined in~\eqref{eq:tdef},
\begin{equation}\label{eq:keydef}
\Phi_k(G)\ge kt.
\end{equation}
\end{lemma}

\begin{proof}
First take $T\in\cE$ and put $u=|V(T)|$. Since its internal degrees
are even and $k$ is odd,
$d_T(v)\le\min\{k-1,u-1\}$ for every $v\in V(T)$. Hence
\begin{equation}\label{eq:internaldef}
k|V(T)|-2e(T)
\ge u\bigl(k-\min\{k-1,u-1\}\bigr)\ge k.
\end{equation}
Indeed, if $u\ge k$, the middle expression is at least $u\ge k$;
if $1\le u\le k$, it is at least
$u(k-u+1)=k+(u-1)(k-u)\ge k$.

For any $\cX\subseteq\cE$, let
$U=\bigcup_{T\in\cX}V(T)$ and $Z=N_B(\cX)\subseteq A$.
The components in $\cX$ have no edges to one another, and every
edge leaving $U$ ends in $Z$: there are no edges from $D$ to $R$,
and different components of $G[D]$ have no edges between them.
Therefore
\begin{align}
\Phi_k(G)
&\ge\sum_{v\in U}\bigl(k-d_G(v)\bigr)\notag\\
&=\sum_{T\in\cX}\bigl(k|V(T)|-2e(T)\bigr)-e_G(U,Z)\notag\\
&\ge k|\cX|-k|Z|
 =k\bigl(|\cX|-|N_B(\cX)|\bigr).\label{eq:boundarycount}
\end{align}
The first inequality uses nonnegativity of the degree deficiencies;
the last uses~\eqref{eq:internaldef} and
$e_G(U,Z)\le\sum_{z\in Z}d_G(z)\le k|Z|$.
Taking the maximum over $\cX\subseteq\cE$ proves~\eqref{eq:keydef}.
\end{proof}

\section{A structural bisection bound}\label{sec:structural}

The next proposition is independent of any maximum-degree bound.
\begin{proposition}\label{prop:structural}
Let $G$ have even order $n$ and $m$ edges, and let $t(G)$ be defined
by~\eqref{eq:tdef}. Then
\begin{equation}\label{eq:structural}
\bis(G)\ge\frac m2+\frac{n-t(G)}4.
\end{equation}
\end{proposition}

\begin{proof}
Choose a matching $J$ of $B$ saturating $A$ and leaving exactly
$t=t(G)$ components of $\cE$ uncovered, as guaranteed by
Lemma~\ref{lem:preferred}. Let $g$ be the number of components of
$\cO$ not covered by $J$.

Lift $J$ to a maximum matching $M$ of $G$ as described after
Theorem~\ref{thm:GE}. In every component $T\in\cO$ not covered
by $J$, choose a vertex $w_T$ of odd internal degree and use a
perfect matching of $T-w_T$. In every uncovered component of
$\cE$, expose an arbitrary vertex. The set $W$ of vertices exposed
by $M$ consequently contains one vertex from each uncovered
component and no other vertices. Put $s=|M|$, and let $W_{\cO}$
be the exposed vertices chosen in components of $\cO$. Then
\begin{equation}\label{eq:counts}
|W_{\cO}|=g,\qquad |W\setminus W_{\cO}|=t,\qquad n=2s+g+t.
\end{equation}

We claim that any pair $\{w,w'\}$ with $w\in W_{\cO}$ and
$w'\in W\setminus\{w\}$ satisfies~\eqref{eq:oddadjacency}.
Let $T$ be the component containing $w$. The matching edges inside
$T$ cover $T-w$. As $d_T(w)$ is odd, at least one of those matching
edges inside $T$, say $xy$, has exactly one end-vertex adjacent to $w$; otherwise counting the neighbors of $w$ in these pairs would give an even degree. The vertex $w'$ is in a different component of $G[D]$, and hence is adjacent to neither $x$ nor $y$. Thus $e_G(\{w,w'\},\{x,y\})=1,$ which proves the claim.

If $g$ is even, pair all vertices of $W_{\cO}$ among themselves.
If $g$ is odd, then~\eqref{eq:counts} and the evenness of $n$ imply
that $t$ is odd, so $t\ge1$. Pair one vertex of $W_{\cO}$ with
one vertex of $W\setminus W_{\cO}$, and pair the remaining vertices
of $W_{\cO}$ among themselves. In either case, this gives
$q=\ceil{g/2}$ disjoint pairs satisfying~\eqref{eq:oddadjacency}.
Lemma~\ref{lem:pairing} now yields
\[
\bis(G)\ge\frac m2+\frac{s+q}{2}
\ge\frac m2+\frac{s}{2}+\frac g4
=\frac m2+\frac{n-t}{4},
\]
as required.
\end{proof}

\section{Proof of the main theorem}\label{sec:mainproof}

\begin{proof}[Proof of Theorem~\ref{thm:main}]
Suppose first that $k$ is odd. We may
assume that $n$ is even: otherwise, add one isolated vertex, apply
the even-order assertion, and delete that vertex from the resulting
bisection. This changes neither $m$ nor the maximum-degree bound
and loses no crossing edges.

By Lemma~\ref{lem:deficiency} and Proposition~\ref{prop:structural},
\[
\bis(G)\ge\frac m2+\frac{n-t(G)}4
\ge\frac m2+\frac14\left(n-\frac{\Phi_k(G)}k\right)
=\frac m2+\frac{m}{2k}
=\frac{k+1}{2k}\,m.
\]
This proves the odd case. If $k$ is even, apply the odd case with
$k+1$, using $\Delta(G)\le k+1$, to obtain
$\bis(G)\ge\frac{k+2}{2(k+1)}m$.

Finally, $K_{k+1}$ has $k(k+1)/2$ edges, and every bisection of it
has $\floor{(k+1)^2/4}$ crossing edges. Consequently,
\[
\frac{\bis(K_{k+1})}{|E(K_{k+1})|}
=\frac{\floor{(k+1)^2/4}}{k(k+1)/2}
=\begin{cases}
\displaystyle\frac{k+1}{2k},&\text{if $k$ is odd},\\[6pt]
\displaystyle\frac{k+2}{2(k+1)},&\text{if $k$ is even}.
\end{cases}
\]
Both coefficients are therefore best possible.
\end{proof}

\section{Concluding remarks}\label{sec:conclusion}

Theorem~\ref{thm:main} proves the unweighted case of the conjecture
of Gerke, Gutin, Yeo and Zhou~\cite{GGYZ}. The exact coefficients
hold without regularity, connectedness or a positive minimum degree.
The central estimate is~\eqref{eq:boundarycount}: each even-degree
factor-critical component supplies at least $k$ units of internal
degree deficiency, while each vertex of its external neighborhood
can account for at most $k$ boundary edges. Hall deficiency therefore
measures a loss that the total degree deficiency can absorb.

The argument is specific to unweighted graphs. In
Lemma~\ref{lem:pairing}, an odd integer number of edges guarantees a
gain of at least $1/2$ over half the newly counted edges. Arbitrary
nonnegative edge weights do not provide the same uniform gain,
and the degree-deficiency estimate controls edge counts rather than
total edge weight. Thus this proof does not establish the full
edge-weighted conjecture.

\end{document}